\documentclass[12pt]{article} 
\usepackage[letterpaper,left=2.5cm,top=2.5cm,right=2.5cm,bottom=2.5cm]{geometry}  
\usepackage[none]{hyphenat}
\usepackage{latexsym} 
\usepackage{layout}
\usepackage{float}
\usepackage{textcomp}
\usepackage{amstext}
\usepackage{multicol} 
\usepackage{longtable} 
\usepackage{lscape} 
\usepackage{multirow} 
\usepackage{array} 
\usepackage{amssymb}
\usepackage{amsmath} 
\usepackage{amsthm}
\usepackage[T1]{fontenc} 
\usepackage{graphicx}
\usepackage{verbatim}

\newtheorem{theorem}{Theorem}[section]
\newtheorem{lemma}[theorem]{Lemma}

\newtheorem{corollary}[theorem]{Corollary}

\theoremstyle{plain}
\newtheorem{prop}[theorem]{Propersition}

\usepackage[pdftex]{color}
\usepackage[pdftex,colorlinks=true,linkcolor=blue,citecolor=black,urlcolor=cyan]{hyperref}

\newlength{\LL}

\title{ \Huge \bf New Approach for Interior Regularity of Monge-Amp\`{e}re Equations}
\author{Chen Ruosi \footnote{crs22@mails.tsinghua.edu.cn}
~,~Zhou Xingchen  \footnote{zxc3zxc4zxc5@stu.xjtu.edu.cn}
~~ \\ {\it Department of Mathematical Sciences} \\ {\it Tsinghua University}
}

\date{\today}

\begin{document}

\maketitle

\hrulefill

\begin{abstract}
By developing an integral approach, we present a new method for the interior regularity of strictly convex solution of the Monge-Amp\`{e}re equation $\det D^2 u = 1$. 
\end{abstract}

\section{Introduction}

In this article, we derive an interior regularity for the Monge-Amp\`{e}re equation 
\begin{equation}
    \det D^2 u = 1 
    \label{eq:MA}
\end{equation}
in general dimension. 
This equation arises naturally from geometric problems such as Weyl and Minkowski problem, and from applied mathematics such as optimal transportation. It has been intensively studied since the last century, and one can refer to \cite{Caffa1990AnnalsW2p,Yau1977CPAM_MAregularity,FigalliMAbook} for more information.

\

Very recently, through a doubling argument in terms of the extrinsic distance function on the maximal Lagrangian submanifold, Shankar-Yuan \cite{YuanShankar2024doubling} gave a new proof for the interior regularity of this equation. 
In this paper, we derive an interior $C^2$ estimate using an integral method. The regularity is then a direct corollary, see \cite{CNS1984CPAM_Dirichlet,GTbook} and references therein for more details.

\

We will denote by $C$ a universal constant depending only on $n, \|u\|_{C^{0,1}(\Omega)}, \|w\|_{C^{0,1}(\Omega)}$, which may change line by line. We state our result in the following:

\begin{theorem}
\label{Thm_MA}
Let $u\in C^4(\Omega)$ be a convex solution of \eqref{eq:MA} on a domain $\Omega \subset \mathbb{R}^n$. Suppose that there is a function $w\in C^2(\Omega)$ such that $w>u$ in $\Omega$ and $ w = u $ on $\partial\Omega$, and there exists some universal constant $C$ such that
\begin{equation}
\label{ineq:Trace_w}
    u^{ij} w_{ij} \geq -C, 
\end{equation}
where $[u^{ij}]$ is the inverse of $D^2 u$. Then 
$$
(w-u)^{(2^n+1)(2n+1)} |D^2 u| \leq C(n,\|u\|_{C^{0,1}(\Omega)}, \|w\|_{C^{0,1}(\Omega)}).
$$
\end{theorem}

When we take $w=0$, condition \eqref{ineq:Trace_w} holds naturally, then Theorem \ref{Thm_MA} gives the familiar Pogorelov type estimate
$$ (-u)^{C(n)} |D^2 u| \leq C(n, \|u\|_{C^{0,1}(\Omega)}). $$
If we further assume $\Omega$ is normalized, say $B_r(0) \subset \Omega \subset B_R(0)$, then we have interior estimate
$$ \| D^2 u \|_{\Omega_{\epsilon}} \leq C(n, \epsilon, r, R, \|u\|_{C^{0,1}(\Omega)}), $$
where $\epsilon >0 $ and $\Omega_{\epsilon} = \{ x\in \Omega: \operatorname{dist} (x,\partial \Omega) > \epsilon \}.$

\begin{corollary}
    Let $u$ be a strictly convex solution of \eqref{eq:MA} on a domain $\Omega \subset \mathbb{R}^n$ for $n\geq 2.$ Then $u$ is smooth inside $\Omega$.
\end{corollary}

Interior $C^2$ estimate plays a key role in elliptic equation from time to time. Using isothermal coordinates, Heinz \cite{heinz1959elliptic} established interior $C^2$ estimate for equation $\det D^2 u = f $ in dimension $2$. New proofs for estimates in two dimension case were found by Chen-Han-Ou \cite{ChenHanOu2016MA} and Liu \cite{Liu_2021_interior_C2_MA}. Pogorelov and Chou-Wang \cite{ChouWang2001variational} derived interior Hessian estimates for solutions with certain strict convexity constraints to equation \eqref{eq:MA} and $\sigma_k$ equations ($k \geq 2$) in general dimension. 
We also mention Hessian estimate for solutions to $\sigma_2$ equation by Guan-Qiu \cite{guan2019interior}, Qiu \cite{qiu2024interiorHessian} and Shankar-Yuan \cite{Yuanshankar2023hessian,Yuan2009CPAM_n=3}.

\

The idea of our proof is as follows. The function $b = \det(I+D^2 u) $ satisfies a Jacobi inequality from \cite{AmonotonicityYuan2023}. 
Together with some cutoff function $\varphi$, we prove the boundary Jacobi inequality for $\varphi b$ in Section 2. In Section 3, we show that this boundary Jacobi inequality results in a Pogorelov type $W^{2,n+1}$ estimate, which states that $L^{n+1}$ norm of $\varphi b$ is bounded. Finally, in Section 4, following the spirit of Moser's iteration, we prove that the $L^{\infty}$ norm of $\varphi^{2n+1} b$ is controlled by $L^{n+1}$ norm of $\varphi b$, which leads to the conclusion of the theorem.

\section{Boundary Jacobi inequality}

As in Yuan \cite{AmonotonicityYuan2023}, we consider the Lagrangian graph $\mathcal{M}_a=\{(x,Du(x))\}\subset (\mathbb R^n\times \mathbb{R}^n, dxdy)$, and denote the induced metric $g=[g_{ij}]_{n\times n}=D^2u$ with its inverse $[g^{ij}]_{n\times n}$. The Laplace-Beltrami operator  takes the non-divergence form
$\Delta_g=g^{ij}\partial_{ij}$. The inner product with respect to metric $g$ is $\langle \nabla_g v, \nabla_g w \rangle = \sum_{i,j=1}^n g^{ij} v_i w_j$, and in particular, $|\nabla_g v|^2 = \sum_{i,j=1}^n g^{ij}v_i v_j.$

In view of Lemma 0.1 in \cite{AmonotonicityYuan2023}, we have the following Jacobi inequality 
\begin{equation}
    \label{ineq:MA_Jacobi}
    \Delta_g b\ge (1+\frac{1}{2n})|\nabla_g b|^2b^{-1},
\end{equation}
where $b=\det(I_n +D^2u)$. This leads to a boundary Jacobi inequality.
\begin{lemma} 
\label{lemma:MA_Boundary_Jacobi}
Let $\eta(t)=(t^+)^{\beta}$, $t\in \mathbb{R}$ with $\beta \geq 2n+1,$ and $\varphi(x) = \eta((w-u)(x))$. Then 
\begin{align*}
\Delta_g (\varphi b)
&\ge \eta'(w-u) \Delta_g(w-u) b.
\end{align*}

\end{lemma}
\begin{proof}
Using \eqref{ineq:MA_Jacobi}, we get
\begin{align*}
\Delta_g(\varphi b)
&= b\Delta_g\varphi+\varphi\Delta_g b + 2g^{ij}\varphi_i b_j \\
&\ge[\Delta_g\varphi-\frac{2n}{2n+1} \varphi^{-1}|\nabla_g\varphi|^2]b
+[\Delta_g b - \frac{2n+1}{2n} b^{-1}|\nabla_g b|^2]\varphi \\
&\ge [\Delta_g\varphi-\frac{2n}{2n+1} \varphi^{-1}|\nabla_g\varphi|^2]b.
\end{align*}
At the points where $w-u>0$, we have
\begin{align*}
&\Delta_g \varphi =\eta''(w-u)|\nabla_g (w-u)|^2+\eta'(w-u)\Delta_g(w-u),\\
&|\nabla_g \varphi |^2= \eta'(w-u)^2 |\nabla_g (w-u)|^2.
\end{align*}
From our choice of $\beta$, we get
\begin{eqnarray*}
& &\Delta_g\varphi-(1+\frac{1}{2n})^{-1}\varphi^{-1}|\nabla_g\varphi|^2\\
&=&\left[\eta''(w-u)-\frac{2n}{2n+1} \frac{\eta'(w-u)^2}{\eta(w-u)}\right]|\nabla_g (w-u)|^2+\eta'(w-u)\Delta_g(w-u)\\
&\ge &\eta'(w-u) \Delta_g(w-u),
\end{eqnarray*}
which leads to $\Delta_g(\varphi b) \geq \eta'(w-u)\Delta_g(w-u)b$. 
\end{proof}

\section{Pogorelov type \texorpdfstring{$W^{2,p}$}{} estimate }

\begin{lemma}
\label{lemma:PogoW2,p}
Define $\varphi$ as in Lemma \ref{lemma:MA_Boundary_Jacobi} and take $\beta = 2^n+1$. Then for any $p\in \mathbb{N}^+$, we have Pogorelov type $W^{2,p}$ estimate
    \begin{align}
    \label{eq:MA_W2p}
    \int_{\Omega}\varphi^{p-\alpha} b^p dx \leq C(n,p,\|w\|_{C^{0,1}(\Omega)},\|u\|_{C^{0,1}(\Omega)}),
    \end{align}
    where $\alpha = \frac{1}{\beta}$.
\end{lemma}

\begin{proof}
Let $p\ge 1$ be an integer, we choose $[\varphi \det(I_n +D^2u)]^p:=(\varphi b)^p = [((w-u)^+)^{\beta} b]^p$ as a test function. From Lemma \ref{lemma:MA_Boundary_Jacobi}, we integrate by parts,
\begin{equation*}
    p\int_{\Omega}|\nabla_g(\varphi b)|^2(\varphi b)^{p-1} dv_g \leq -\int_{\Omega} \eta'(w-u) \Delta_g (w-u) b^{p+1}\varphi ^pdv_g. 
\end{equation*}
Since $ \det D^2 u = 1 $, we have $dv_g = dx$. Notice that $\Delta_g u=g^{ij}u_{ij}=n$ and that $-\Delta_gw\leq C(n) $ from \eqref{ineq:Trace_w}, we have $ -\Delta_g (w-u)\leq C(n).$ Since $ \varphi = ((w-u)^+)^\beta$, we have $\eta'(w-u) = \beta \varphi^{1-\alpha} \geq 0,$ where $\alpha = \frac{1}{\beta}$. This yields
$$
-\eta'(w-u)\Delta_g (w-u)\leq C(n) \varphi^{1-\alpha},
$$
and therefore
\begin{equation}
    \label{ineq:MA_nabla_g}
    p\int_{\Omega}|\nabla_g(\varphi b)|^2(\varphi b)^{p-1} dx \leq C \int_{\Omega} \varphi^{p+1-\alpha} b^{p+1} dx.
\end{equation}
Split $b$ into summation of the fundamental symmetric functions, 
\begin{align*}
    &\int_{\Omega}\varphi^{p+1-\alpha}b^{p+1}\,dx = \int_{\Omega}\varphi^{p+1-\alpha} b^p (1+\sigma_1+\cdots+\sigma_{n-1}+1) \,dx.
\end{align*}
For $1\leq k \leq n-1 $, using the relation $\sum_{i,j} \frac{\partial \sigma_k}{\partial u_{ij}} u_{ij} = k\sigma_k$ and integration by parts,
\begin{equation}
    \begin{aligned}
        \int_{\Omega}(\varphi b)^{p} \varphi^{1-\alpha} \sigma_k dx 
        & = & &\frac{1}{k} \int_{\Omega} \sum_{i,j} \frac{\partial \sigma_k}{\partial u_{ij}} (\varphi b) ^{p} \varphi^{1-\alpha} D_i(u_j) \,dx \\
        & = & &- \frac{p}{k}\int_{\Omega} \sum_{i,j} \frac{\partial \sigma_k}{\partial u_{ij}} D_i(\varphi b)(\varphi b)^{p-1} \varphi^{1-\alpha} D_ju dx \\
        & & &- \frac{\beta-1}{k} \int_{\Omega} \sum_{i,j} \frac{\partial \sigma_k}{\partial u_{ij}} D_i(w-u) \varphi^{1-2\alpha} (\varphi b)^{p} D_ju dx ,
    \end{aligned}
    \label{eq:sigma_k}
\end{equation}
where we use divergence free property $\sum_i D_i(\frac{\partial \sigma_k}{\partial u_{ij}}) = 0$ in the second equality. 
Pointwisely, we can choose a coordinate such that $D^2u$ is diagonal, then the first term in \eqref{eq:sigma_k} can be estimated by
\begin{align*}
\varphi^{1-\alpha}\left|\frac{\partial \sigma_k}{\partial u_{ij}} D_i(\varphi b)D_ju\right|
& = \varphi^{1-\alpha}\left|\frac{\partial \sigma_k}{\partial\lambda_i}D_i(\varphi b)D_iu\right| = \varphi^{1-\alpha}\left|(\frac{\partial \sigma_k}{\partial\lambda_i}\sqrt{\lambda_i}) D_iu \frac{D_i(\varphi b)}{\sqrt{\lambda_i}}\right| \\
&\leq \delta |\nabla_g(\varphi b)|^2 + C(\delta,\|Du\|_{L^\infty}) \varphi^{2-2\alpha} b\sigma_{k-1}
\end{align*}
for some constant $\delta(n)>0 $, where we use the relation $|\frac{\partial \sigma_k}{\partial\lambda_i}\sqrt{\lambda_i}|^2 \leq \sigma_k\sigma_{k-1}\leq b\sigma_{k-1}$. 
The second term in \eqref{eq:sigma_k} can be estimated similarly.
So we get
\begin{equation}
    \label{eq:MA_induction_sigmak}
    \begin{aligned}
        \int_{\Omega}\varphi^{p+1-\alpha}b^{p+1}dx 
        & \leq \sum_{k=1}^{n-1} \int_{\Omega}\varphi^{p+1-\alpha} b^p  \sigma_k + C \int_{\Omega}  \varphi^{p+1-\alpha} b^p\\
        & \leq p C \sum_{k=1}^{n-1}\int_{\Omega}\varphi^{p+1-2\alpha} b^p \sigma_{k-1} dx+p\delta\int_{\Omega}|\nabla_g(\varphi b)|^2(\varphi b)^{p-1} dx + C \int_{\Omega}\varphi^{p+1-\alpha} b^p \\
        & \leq p C \sum_{k=1}^{n-2}\int_{\Omega} \varphi^{p+1-2\alpha}b^p \sigma_{k} dx + p\delta\int_{\Omega}|\nabla_g(\varphi b)|^2(\varphi b)^{p-1} dx + C \int_{\Omega} \varphi^{p+1-2\alpha} b^p
    \end{aligned}
\end{equation}
Thus for $\delta(n)$ small, using \eqref{ineq:MA_nabla_g}, we have 
\begin{equation}
\label{ineq:induction_1}
    \int_{\Omega}\varphi^{p+1-\alpha} b^{p+1}dx \leq p C \sum_{k=1}^{n-2}\int_{\Omega}\varphi^{p+1-2\alpha} b^p \sigma_{k} dx + C \int_{\Omega} \varphi^{p+1-2\alpha} b^p.
\end{equation}
Similar to the procedures \eqref{ineq:MA_nabla_g}-\eqref{ineq:induction_1}, we go through these steps for at most $n$ times,
\begin{equation}
\label{ineq:induction_2}
\begin{aligned}
    \sum_{k=1}^{n-2}\int_{\Omega}\varphi^{p+1-2\alpha} b^p\sigma_{k} dx & \leq p C \sum_{k=1}^{n-3}\int_{\Omega}\varphi^{p+1-4\alpha} b^p \sigma_{k} dx + C \int_{\Omega} \varphi^{p+1-4\alpha} b^p \\
    & \leq \cdots \\
    & \leq C \int_{\Omega} \varphi^{p+1-2^n \alpha} b^p.
\end{aligned}
\end{equation}
Therefore, from our choice of $\alpha$, we derive an iteration formula
\begin{eqnarray*}
    \int_{\Omega}\varphi^{p+1-\alpha} b^{p+1} \leq C \int_{\Omega} \varphi^{p+1-2^n \alpha} b^p \leq C \int_{\Omega} \varphi^{p- \alpha} b^p.
\end{eqnarray*}
Iterating by $p$ times, we get
\begin{eqnarray*}
    \int_{\Omega} \varphi^{p+1-\alpha}b^{p+1} dx \leq C \int_{\Omega}\varphi^{1-\alpha} b .
\end{eqnarray*}
Split $b$ again, we have $\int_{\Omega} \varphi^{1-\alpha} b dx \leq C + \sum_{k=1}^{n-1} \int_{\Omega} \varphi^{1-\alpha} \sigma_k$. For $1\leq k \leq n-1$, 
\begin{eqnarray*}
    \int_{\Omega} \varphi^{1-\alpha} \sigma_k = \frac{1}{k} \int_{\Omega} \varphi^{1-\alpha} \frac{\partial \sigma_k}{\partial u_{ij}} D_iu_j = -\frac{\beta-1}{k} \int_{\Omega} \varphi^{1-2\alpha} D_i(w-u) \frac{\partial \sigma_k}{\partial u_{ij}} u_j.
\end{eqnarray*}
This yields
\begin{eqnarray*}
    \int_{\Omega} \varphi^{1-\alpha} \sigma_k \leq C \int_{\Omega} \varphi^{1-2\alpha} \sigma_{k-1}.
\end{eqnarray*}
Therefore, repeating this estimate by $n-1$ times, we get
\begin{eqnarray*}
    \int_{\Omega} \varphi^{1-\alpha} b  \leq C + C \sum_{k=1}^{n-2} \int_{\Omega} \varphi^{1-2\alpha} \sigma_k \leq \cdots \leq C + C \int_{\Omega} \varphi^{1-n\alpha} \leq C,
\end{eqnarray*}
where we use $1-n\alpha \geq 1-2^n \alpha \geq 0$ from our choice of $\alpha$.
\end{proof}

\section{From \texorpdfstring{$W^{2,p}$}{} estimate to \texorpdfstring{$W^{2,\infty}$}{} estimate}

\begin{prop} 
Define $\varphi$ as in Lemma \ref{lemma:PogoW2,p}, we have 
$$
    \|\varphi^{2n+1} b\|_{L^{\infty}(\Omega) } \leq C \int_{\Omega} \varphi^{n+1-\alpha} b^{n+1}dx ,
$$
where $ C = C(n, \|u\|_{C^{0,1}(\Omega)}, \|w\|_{C^{0,1}(\Omega)})$.
\end{prop}

\begin{proof}
Let $p\ge 1,q\ge 2$ be a pair of positive constants satisfying $ \frac{q}{p} \leq 2n+1$. 
Multiply \eqref{ineq:MA_Jacobi} by $\varphi^q b^p $ and integrate on $\Omega $, we have
$$
\int_{\Omega} \varphi^q b^p \Delta_g b \,dv_g \ge 0.
$$
Integrating by parts the above inequality, we get
\begin{equation}
    -p\int_{\Omega } g^{ij}b_i b_j b^{p-1} \varphi^q  \,dv_g - q \int_{\Omega_{\epsilon-\rho}} g^{ij}b_i \varphi_j b^p \varphi^{q-1}  \,dv_g \geq 0.
\end{equation}
Recall that $dv_g = dx$. By Young's inequality, we have
\begin{eqnarray*}
    p\int_{\Omega} g^{ij}b_i b_j b^{p-1} \varphi^q \,dx &\leq& q\left( \epsilon_1(n) \int_{\Omega} g^{ij}b_i b_j b^{p-1}\varphi^{q} \,dx + C \int_{\Omega} g^{ij}\varphi_i \varphi_j b^{p+1} \varphi^{q-2} \,dx \right).
\end{eqnarray*}
Since $q \leq (2n+1)p$ and $ p\geq 1$, we can choose $\epsilon_1(n) $ small and get
\begin{eqnarray*}
    \int_{\Omega} g^{ij}b_i b_j b^{p-1} \varphi^q \,dx \leq C(n) \int_{\Omega} g^{ij}\varphi_i \varphi_j b^{p+1} \varphi^{q-2} \,dx.
\end{eqnarray*}
Pointwisely, we may assume $D^2 u$ is diagonal, then $ \delta_{ij} b^{-1} \leq g^{ij} = \delta_{ij} \lambda_i^{-1} \leq \delta_{ij} b$. Therefore
\begin{eqnarray*}
    \int_{\Omega} g^{ij}b_i b_j b^{p-1} \varphi^q \,dx \geq \int_{\Omega} |Db|^2 b^{p-2} \varphi^q \, dx,
\end{eqnarray*}
and
$$ \int_{\Omega} g^{ij}\varphi_i \varphi_j b^{p+1} \varphi^{q-2}\,dx \leq \int_{\Omega} |D\varphi|^2 b^{p+2} \varphi^{q-2} \,dx.$$
Since $b\geq 1$ and $\varphi \leq C$, this leads to
\begin{eqnarray*}
    \int_{\Omega} |Db|^2 b^{p-2} \varphi^q \, dx \leq C \int_{\Omega} \varphi^{q-2} b^{p+2} \,dx,
\end{eqnarray*}
and then
\begin{eqnarray*}
    \int_{\Omega} |D(\varphi^{\frac{q}{2}} b^{\frac{p}{2}})|^2 \,dx &=& \int_{\Omega} \left|\frac{p}{2}b^{\frac{p}{2}-1}\varphi^{\frac{q}{2}} Db + \frac{q}{2} b^{\frac{p}{2}} \varphi^{\frac{q}{2}-1} D\varphi \right|^2 \,dx \\
    &\leq & 2\int_{\Omega} \frac{p^2}{4} b^{p-2} \varphi^q |Db|^2 \,dx + 2\int_{\Omega} \frac{q^2}{4} b^p \varphi^{q-2} |D\varphi|^2 \,dx \\
    &\leq& Cp^2 \int_{\Omega} \varphi^{q-2} b^{p+2}\,dx,
\end{eqnarray*}
together with
$$\int_{\Omega} (\varphi^{\frac{q}{2}} b^{\frac{p}{2}})^2 \,dx \leq C \int_{\Omega} \varphi^{q-2} b^{p+2}\,dx \leq Cp^2 \int_{\Omega}\varphi^{q-2} b^{p+2} \,dx. $$ 
When $n>2$, we have $\varphi^{\frac{q}{2}} b^{\frac{p}{2}}  \in W^{1,2}(\Omega) \hookrightarrow L^{\frac{2n}{n-2}}(\Omega) $ by Sobolev embedding theorem and therefore
\begin{eqnarray*}
    \int_{\Omega} \varphi^{\frac{nq}{n-2}} b^{\frac{np}{n-2}} \,dx \leq \left(C p^2 \int_{\Omega} \varphi^{q-2} b^{p+2} \,dx\right) ^{\frac{n}{n-2}}.
\end{eqnarray*}
Denote $\gamma=\frac{n}{n-2}$, we fix an integer $k_0\ge\frac{\ln n}{\ln\gamma}$, and take $p_0 = \gamma^{k_0}$, $q_0 = (2n+1)p_0$ to initiate the iteration. For $k = 1,\cdots,k_0$, let
\begin{eqnarray*}
    p_{k} &=& \gamma^{-1} p_{k-1}+2=\gamma^{-k}p_0+2\sum^{k-1}_{i=0}\gamma^{-i}, \\
    q_{k} &=& \gamma^{-1} q_{k-1}-2=\gamma^{-k}q_0-2\sum^{k-1}_{i=0}\gamma^{-i}.
\end{eqnarray*}
We can check that $p_k \geq n$ and $q_k \ge 2 $ satisfy $\frac{q_k}{p_k} \leq 2n+1$ all along.
Then we can rewrite the formula as
$$
\int_{\Omega} \varphi^{q_{k-1}} b^{p_{k-1}}\,dx \leq \left[C (p_k-2)^2 \int_{\Omega} \varphi^{q_{k}} b^{p_k} \,dx\right]^{\gamma}.
$$
From the expression of $p_k$, we have $p_k-2 \leq n\gamma^{-k} p_0 $. 
Therefore, an iteration process implies
\begin{eqnarray*}
\int_{\Omega} \varphi^{q_{0}} b^{p_0}\,dx &\leq& C^{\sum^{k_0}_{k=1}\gamma^k}  \prod^{k_0}_{k=1}(p_k-2)^{2\gamma^k}\left(\int_{\Omega} \varphi^{q_{k_0}} b^{p_{k_0}}\,dx\right)^{\gamma^{k_0}}\\
&\leq& C^{\sum^{k_0}_{k=1}\gamma^k} \prod^{k_0}_{k=1}(n\gamma^{-k}p_0)^{2\gamma^k} \left(\int_{\Omega} \varphi^{q_{k_0}}b^{p_{k_0}}\,dx\right)^{\gamma^{k_0}} \\
& =& C^{\sum^{k_0}_{k=1}\gamma^k}  \gamma^{\sum^{k_0}_{k=1}(k_0-k)2\gamma^k} \left(\int_{\Omega} \varphi^{q_{k_0}}b^{p_{k_0}}\,dx\right)^{p_0}.
\end{eqnarray*}
Notice that $\frac{1}{\gamma^{k_0}} \sum^{k_0}_{k=1}\gamma^k $ and $ \frac{1}{\gamma^{k_0}} \sum^{k_0}_{k=1}(k_0-k)\gamma^k $ can both be bounded by some constant $C(n)$. So we get
\begin{eqnarray*}
\|\varphi^{2n+1}b\|_{L^{p_0}(\Omega)} = \left(\int_{\Omega} \varphi^{q_{0}} b^{p_0} \,dx \right)^\frac 1{p_0} \leq C \int_{\Omega} \varphi^{q_{k_0}} b^{p_{k_0}} \,dx .
\end{eqnarray*}
From the choice of $p_0, q_0 $ and $k_0$, we can verify that $p_{k_0} \leq n+1 $ and $q_{k_0} \geq n+1 > n+1-\alpha$.
So the above estimate implies 
$$ \|\varphi^{2n+1} b\|_{L^{p_0} (\Omega)} \leq C \int_{\Omega} \varphi^{n+1-\alpha}b^{n+1} \,dx. $$
Let $k_0\rightarrow +\infty$, that is $p_0 \to +\infty$, we have the conclusion.

When $n=2$, take $\gamma>1$ to be any fixed number, for example $\gamma=2$, then $W^{1,2}\hookrightarrow L^{\gamma}$. Following the above iteration process, one can still obtain
$$ \| \varphi^{2n+1} b\|_{L^{\infty}(\Omega) } \leq C \int_{\Omega} \varphi^{n+1-\alpha} b^{n+1} \,dx .$$
This completes the proof of the proposition.
\end{proof}

\bibliography{References}
\bibliographystyle{acm}

\end{document}